\documentclass{amsart}
\usepackage[T1]{fontenc}
\usepackage{bm}
\usepackage{microtype} 

\usepackage[
  style=ext-alphabetic,
  backend=biber,
  doi=false,
  isbn=false,
  maxalphanames=99,
  maxnames=99,
  articlein=false
]{biblatex}

\DeclareFieldFormat{language}{}
\renewbibmacro*{language}{}

\AtEveryBibitem{%
  \clearfield{language}%
  \ifboolexpr{
    test {\ifentrytype{article}}
    or test {\ifentrytype{book}}
    or test {\ifentrytype{incollection}}
  }
    {\clearfield{url}}
    {}%
  \iffieldundef{eprint}
    {}
    {\clearfield{url}}%
}

\DefineBibliographyStrings{english}{page={}, pages={}}
\DefineBibliographyStrings{french}{page={}, pages={}}
\DefineBibliographyStrings{german}{page={}, pages={}}

\usepackage{tikz}
\usepackage{tikz-cd}
\usepackage{quiver}
\usepackage{adjustbox}
\usepackage{enumitem}
\usepackage[normalem]{ulem}
 \usepackage{stmaryrd}
 \usepackage{mathtools}

 \newcommand{\Rbar}{\overline{R}}

\renewcommand{\mathbb}{\mathbf}
\newcommand{\chibar}{\overline{\chi}}
\newcommand{\psibar}{\overline{\psi}}

\usepackage{amssymb,amsmath,amsfonts,amsthm,epsfig,amscd}
\usepackage[all,cmtip,poly]{xy}
\usepackage{color}

\newcommand{\red}{\operatorname{red}}

\newcommand{\univ}{{\operatorname{univ}}}

\usepackage{ifpdf}
\ifpdf \usepackage[bookmarksopen,bookmarksdepth=4]{hyperref} \fi

\usepackage{chngcntr}
\counterwithin{equation}{subsection}

\makeatletter
\let\oldsubsubsection\subsubsection

\renewcommand{\subsubsection}{\@ifstar{\subsubsectionstar}{\newsubsubsection}}

\newcommand{\subsubsectionstar}{\oldsubsubsection*}

\newcommand{\newsubsubsection}[1]{%
  \refstepcounter{equation}%
  \@startsection{subsubsection}{3}%
  {\z@}{.5\linespacing\@plus.7\linespacing}{-.5em}%
  {\normalfont\itshape}{#1}%
}
\makeatother

\newtheorem{thm}[equation]{Theorem}

\newtheorem{cor}[equation]{Corollary}

\newtheorem{prop}[equation]{Proposition}

\theoremstyle{definition}

\newtheorem{defn}[equation]{Definition}

\theoremstyle{remark}
\newtheorem{remark}[equation]{Remark}
\newtheorem{rem}[equation]{Remark}

\newif\iffinalrun
\iffinalrun
\else
\fi

\iffinalrun
  \newcommand{\need}[1]{}
  \newcommand{\mar}[1]{}
\else
  \newcommand{\need}[1]{{\tiny *** #1}}
  \newcommand{\mar}[1]{\marginpar{\raggedright\tiny  
		  fixme
      #1}}
\fi

\iffinalrun
  \newcommand{\finalmar}[1]{}
\else
  \newcommand{\finalmar}[1]{\marginpar{\raggedright\tiny   
      #1}}
\fi

\newcommand{\F}{\FF}

\newcommand{\Q}{\QQ}

\newcommand{\Z}{\ZZ}

\newcommand{\FF}{{\mathbb F}}

\newcommand{\QQ}{{\mathbb Q}}

\newcommand{\ZZ}{{\mathbb Z}}

\newcommand{\cI}{{\mathcal I}}

\newcommand{\cO}{{\mathcal O}}

\newcommand{\cS}{{\mathcal S}}

\newcommand{\cX}{{\mathcal X}}

\newcommand{\Fbar}{\overline{\F}}
\newcommand{\Qbar}{\overline{\Q}}
\newcommand{\Zbar}{\overline{\Z}}

\newcommand{\Fpbar}{\Fbar_p}

\newcommand{\Zpbar}{\Zbar_p}

\newcommand{\Qp}{\Q_p}

\newcommand{\Qpbar}{\Qbar_p}

\DeclareMathOperator{\Gal}{Gal}
\DeclareMathOperator{\GL}{GL}

\DeclareMathOperator{\PGL}{PGL}

\DeclareMathOperator{\Proj}{Proj}

\DeclareMathOperator{\SL}{SL}

\DeclareMathOperator{\Spec}{Spec}

\DeclareMathOperator{\Spf}{Spf}

\newcommand{\rhobar}{\overline{\rho}}

\newcommand{\into}{\hookrightarrow}
\newcommand{\onto}{\twoheadrightarrow}

\newcommand{\loc}{\mathrm{loc}}

\newcommand{\varepsilonbar}{\overline{\varepsilon}}
\newcommand{\rbar}{\overline{r}}

\usepackage{crossreftools}

\title[Prescribed lifts of 2-dimensional representations]{Prescribed lifts of 2-dimensional representations}

\author[M. Emerton]{Matthew Emerton}\email{emerton@math.uchicago.edu}
\address{Department of Mathematics, University of Chicago,
5734 S.\ University Ave., Chicago, IL 60637, USA}

\author[T. Gee]{Toby Gee} \email{toby.gee@imperial.ac.uk} \address{Department of
  Mathematics, Imperial College London,
  London SW7 2AZ, UK}

\author[L. Pan]{Lue Pan}  \email{lp3224@columbia.edu} \address{Department of Mathematics,
Columbia University,
2990 Broadway,
New York, NY 10027, USA}

\author[X. Zhu]{Xinwen Zhu}\email{zhuxw@stanford.edu}
\address{Department of Mathematics, Stanford University, USA}
\thanks{ ME was supported in part by the
  NSF grants 
  DMS-1902307, DMS-1952705, DMS-2201242 and~DMS-2502609. TG was 
  supported in part by an ERC Advanced grant. This project has received funding from the European Research Council
  (ERC) under the European Union’s Horizon 2020 research and
  innovation programme (grant agreement No. 884596). ME and TG were
  both supported in part by the Simons Collaboration on Perfection in Algebra, Geometry and Topology.
LP was supported in part by a Sloan Research Fellowship.
XZ was supported by NSF grant DMS-2200940.
For the purpose of open access, the authors have applied a CC BY public copyright licence to any author accepted manuscript arising from this submission.
 }

\begin{document}

\begin{abstract}Let~$F$ be a totally real field, and let~$p$ be prime.
Under standard Taylor--Wiles hypotheses, we show that an irreducible totally odd Galois representation $\rhobar:\Gal_F\to\GL_2(\Fpbar)$ admits lifts lying on arbitrary prescribed components of local deformation rings, allowing potentially semistable conditions with arbitrary regular Hodge--Tate weights.
\end{abstract}

\maketitle
\setcounter{tocdepth}{1}
\tableofcontents
\section{Introduction}\label{sec:introduction}
\refstepcounter{subsection}
Let $F$ be a totally real field, and let $\Sigma$ be a finite set of finite places of~$F$
containing all places above a fixed prime~$p$.
We write $\Gal_F\coloneq\Gal(\overline{F}/F)$ for the absolute Galois group of~$F$,
and we denote by $\Gal_{F,\Sigma}$ the Galois group of the maximal extension of~$F$ unramified at all finite
places outside~$\Sigma$.

We consider the following local-to-global lifting problem: let~$\cO$ be the
ring of integers in a finite extension of~$\Qp$, with residue field~$\F$, and let
\[
  \rhobar:\Gal_{F,\Sigma}\to\GL_2(\F)
\]
be a continuous representation.
Suppose that we are given a de Rham character~$\psi:\Gal_{F,\Sigma}\to\cO^\times$ lifting~$\det\rhobar$, and local conditions on lifts of~$\rhobar$ at the places in~$\Sigma$.
At places away from~$p$ these conditions are allowed to be any non-empty union of irreducible components of the appropriate fixed-determinant local deformation ring, while at places above~$p$ they are allowed to be any non-empty union of irreducible components of a potentially semistable deformation ring with prescribed regular Hodge--Tate weights.
We write~$\cS$ for this collection of data, and~$R_{\cS}$ for the corresponding global deformation ring; the precise definition is given in Definition~\ref{defn:global-deformation-problem}.

Our main theorem is the following; note that the assumption that~$\psi$ is totally odd is a consequence of~$\rhobar$ being totally odd, unless~$p=2$.

\begin{thm}[Theorem~\ref{thm:main-thm-existence-lifts}]
\label{thm:main-intro}
  Suppose that~$\rhobar$ and~$\psi$ are totally odd, and that
  $\rhobar|_{\Gal_{F(\zeta_p)}}$ is absolutely irreducible.  If~$p=2$, assume that~$\rhobar$ has non-solvable image, and if ~$p=5$, assume  that~$\rhobar$ is not exceptional in the sense of Definition~\emph{\ref{defn:exceptional}}.
 
  Then, for every global deformation
  problem~$\cS$, the ring~$R_{\cS}$ is a finite~$\cO$-algebra of positive
  $\cO$-rank.  In particular, the set of characteristic-zero lifts
  of~$\rhobar$ to representations of~$\Gal_{F,\Sigma}$ of type~$\cS$ is finite and non-empty.
\end{thm}

In short, subject to the hypotheses of the theorem, arbitrary prescribed
non-empty sets of local components occur globally.

\begin{remark}
Although the theorem as we have phrased it (allowing our local conditions to
be unions of irreducible components of local deformation rings)
is equivalent to the analogous version in which we prescribe a single choice
of irreducible component at each prime, we find our phrasing convenient.
For example, imposing a bound on the prime-to-$p$ conductor of lifts of~$\rhobar$ amounts to considering a certain union of local components.  Thus the theorem immediately implies that the set of lifts of~$\rhobar$ satisfying any given bound on the prime-to-$p$ conductor (compatible with the prime-to-$p$ conductor of~$\rhobar$, and the ramification of~$\psi$), as well as any prescribed $p$-adic Hodge-theoretic conditions at~$p$ (again, compatible with~$\psi$), is finite and non-empty.
\end{remark}

In particular, the theorem gives a Galois-theoretic analogue of theorems on level raising and lowering, and on change of weight, for Hilbert modular forms; see~\cite[\S 1.2]{MR4386822} for a historical overview of this connection, an explanation of the basic strategy for proving such results (which goes back to Khare--Wintenberger), and an account of the subsequent
generalisations to Hilbert modular forms (see for example
\cite{KW1,KW2,MR2459302,GeePrescribed}).
Very briefly, the strategy of Khare--Wintenberger is as follows.
One can show that 
$\dim R_{\cS}\ge 1$ via a computation in Galois cohomology; see Proposition~\ref{prop:dim-lower-bound}.  If one also
knows that~$R_{\cS}$ is finite over~$\cO$, this lower bound forces
$R_{\cS}$ to have positive~$\cO$-rank, so that $R_{\cS}\otimes_{\cO} \Qbar_p$
is a non-zero finite $\Qbar_p$-algebra.
 Thus  $\Spec (R_{\cS}\otimes_{\cO} \Qbar_p)$ is a finite set, and its points correspond to the to the set of characteristic-zero lifts in Theorem~\ref{thm:main-intro}.

The crux of the argument is therefore to prove
the finiteness of the global deformation ring~$R_{\cS}$. 
In previous work,
this finiteness is ultimately
obtained from modularity: after a suitable potential modularity argument, the
deformation ring is compared with a Hecke algebra, which is finite over~$\cO$.
For two-dimensional totally odd representations over totally real fields,
the required finiteness in ``parallel weight~$2$'' is already known.  Here ``parallel weight~$2$'' means that the labelled Hodge--Tate weights are $\{0,1\}$ at every embedding; the terminology comes from the fact that the Galois representations associated to Hilbert modular forms of parallel weight~$2$ have these Hodge--Tate weights.  In the form we
need, this finiteness is a theorem of Snowden~\cite{0905.4266} (building on Kisin's modularity lifting theorems~\cite{KisinModularity}), which is recalled below as
Proposition~\ref{prop:pot-BT-deformation-ring-finite}.

The main new point of this paper is that we can deduce the general case of arbitrary regular Hodge--Tate weights from this weight~$2$ case.
We do not use any new modularity lifting theorems to do this; indeed, we do not establish the potential modularity of our lifts.
Instead, we make use of the geometry of the
stack~$\cX$ of \'etale $(\varphi,\Gamma)$-modules~\cite{emertongeepicture}
(or, more precisely, its fixed-determinant variant).
A ``topological'' form of the Breuil--M\'ezard philosophy, established in~\cite{emertongeepicture}, states that
the underlying reduced substacks
of the relevant potentially semistable substacks
are equal to
suitable unions of components of the underlying reduced substack~$\cX_{\red}$ of~$\cX$.
Furthermore, each of these components already occurs
in the underlying reduced substack of a potentially semistable substack of parallel weight~$2$.
(Indeed,
Proposition~\ref{prop:underlying-reduced-tame-weight-0}
shows that every residual two-dimensional local representation with the required
determinant is contained in the special fibre of a
tamely potentially semistable weight~$2$ substack.)

Passing from stacks to local deformation rings via versal maps gives
Corollary~\ref{cor:nilpotent-thickening-comparison-local-deformation-rings}:
modulo~$\varpi$ and up to nilpotent thickening, any potentially
semistable local deformation ring is a quotient of a weight~$2$ local deformation ring.
Tensoring these local comparisons together and using the standard description of global
deformation rings as global-over-local tensor products then gives a surjection
\[
  R_{\cS'}/\varpi\onto R_{\cS}/(\varpi,I),
\]
where~$\cS'$ is a weight~$2$ deformation problem and~$I$ is nilpotent.  Since
$R_{\cS'}$ is finite over~$\cO$, it follows that~$R_{\cS}/\varpi$ is finite
over~$\F$, and hence that~$R_{\cS}$ is finite over~$\cO$.
As sketched above, the
Galois-cohomological lower bound then completes the proof of the theorem.

\begin{remark}
  In the particular case when $p$ splits completely in~$F$, it is well-known that Theorem~\ref{thm:main-intro} can be deduced from the Breuil--M\'ezard conjecture for~$\operatorname{GL}_2/\mathbf{Q}_p $, following the methods of Kisin's paper \cite{KisinFM}.
  In particular~\cite[Prop.~
  3.7]{MR2869026} proves a very closely related result (under a residual modularity hypothesis, which is easily removed via standard potential modularity arguments).
  
  After finishing writing this paper, we learned that a different proof (again in the case that~$p$ splits completely in~$F$, and under an easily removed residual modularity hypothesis) was given  in
~\cite[Prop.~8.4]{MR3968877}. The proof of this result in ~\cite{MR3968877}
is similar to ours; rather than using the 
general ``topological Breuil--M\'ezard'' statement that we establish using the stack~$\cX$, Hu and Paškūnas prove  
an analogous statement for deformation rings, in the framework of the geometric Breuil--M\'ezard conjecture of~\cite{emertongeerefinedBM} (which is known in the case that $p$ splits completely, {\em cf.}~the various citations
in the proof of~{\em loc.~cit.}).

It should also be noted that the idea of proving the finiteness of a global deformation ring via a comparison of the special fibres of local deformation rings also appears in  Allen and Calegari's paper \cite{MR3294389}, where they compare unramified and ordinary deformation rings.
\end{remark}

\begin{remark}
  The hypothesis that~$\rhobar$ is totally odd is necessary, as shown by Calegari in~\cite{MR2869026}; in particular, he gives an example of a surjective even representation $\overline{\rho} :\Gal_{\Q}\to\SL_2(\F_{11})$ with no de Rham lifts.
  The assumption that the Hodge--Tate weights of our lifts are regular is also necessary. Indeed, arguing as in the proof of ~\cite[Thm.\ 5.1]{MR2869026},  we note that if~$F=\Q$ and $\rho:\Gal_{\Q}\to\GL_2(\Qpbar)$ is potentially semistable with equal Hodge--Tate weights, then the projective image of $\rho|_{I_{\Q_p}}$ is finite, and thus either cyclic, dihedral, $A_4$, $S_4$ or~$A_5$.
  This gives a contradiction if for example~$p>5$ and $\rhobar|_{I_p}\cong
  \begin{pmatrix}
    \chibar & *\\0 &1
  \end{pmatrix}
  $ is a non-split extension with~$\chibar$ of order greater than~$2$.
\end{remark}

\begin{remark}
  The hypotheses in Theorem~\ref{thm:main-intro} are
those that are required to implement
the usual Taylor--Wiles argument in this two-dimensional setting.
  Some minor improvements are probably within reach.
In particular, 
it is possible that the exceptional case when $p=5$ could be handled (for example, using the methods of~\cite{MR3648503}), but we have not
 seriously investigated this. 

In an ideal formulation, one would replace the Taylor--Wiles condition by the precise cohomological condition that guarantees the expected lower bound on the global deformation ring, for instance the appropriate vanishing of the obstruction group or dual Selmer group.
(Without such a condition, the result need not hold; for example, see~\cite[\S 6.4]{Allen2019} for a discussion from the point of view of Galois deformations of an example due to Serre for~$F=\Q$ and~$p=3$, with~$\rhobar$ induced from $\Q(\sqrt{-3})$.)
We restrict throughout, though, to the residually irreducible case and,
indeed, to the Taylor--Wiles case;
 in forthcoming work, we will relate Theorem~\ref{thm:main-intro} to the local-to-global morphisms of moduli stacks of Galois representations discussed in~\cite[\S 9.2]{MR5007752}, and discuss the residually reducible case.
\end{remark}

\begin{rem}
  One could hope to generalise Theorem~\ref{thm:main-thm-existence-lifts}
to other number fields, and to representations valued in reductive groups other than~$\GL_2$.
  As alluded to in the previous remark, at least if~$\rhobar$ is irreducible, we expect that such generalisations should hold if and only if the Galois-cohomological lower bound on the dimension of the global deformation ring is positive (equivalently, is equal to~$1$).
  This requires the usual ``numerical coincidence'' in the Taylor--Wiles method to hold, as described at the beginning of the introduction to~\cite{cht}; in particular, it forces
the number field 
to be totally real, and the representation~$\rhobar$ to be totally odd, which in
turn forces
an appropriate self-duality to hold.
  The difficulty in actually proving results in these settings is the absence of sufficiently powerful automorphy lifting theorems.
(See Remark~\ref{rem:local-n-dimensional} below for a brief explanation of how our local results extend to arbitrary dimension.)

The main higher-dimensional setting in which automorphy lifting theorems are available is
that of 
automorphic forms on unitary groups.
The corresponding Galois representations are valued in the $L$-groups (or $C$-groups) of these unitary groups, 
and can be interpreted as essentially conjugate-self-dual representations of~$\Gal_{E}$, where~$E$ is a totally imaginary CM field.

The automorphy lifting theorems of~\cite{BLGGT} could be applied to extend our results
to this setting {\em if} one knew that (for example) all crystalline representations with Hodge--Tate weights $0,1,\dots,n-1$ are potentially diagonalizable.
This is known for~$n=2$ by~\cite{geekisin}, and for~$n=3$, it was recently proved (for~$p\ge 5$) in~\cite{bartlett2026resolutionsspacescrystallinerepresentations}.
It should then be straightforward to adapt the methods of this paper to prove the corresponding lifting results for~$n=2$ or~$3$, but we have not attempted to do so.
\end{rem}

\begin{remark}
\label{rem:Ramakrishna and FKP}
  There is another, purely Galois-theoretic approach to producing lifts with prescribed local properties at a fixed finite set of places, going back to Ramakrishna~\cite{MR1935843}.
  It requires one to allow additional ramification at auxiliary places to kill elements of Selmer groups, and so does not let one 
prescribe the local behaviour at every place at once.
  However, it is considerably more flexible, and Fakhruddin--Khare--Patrikis~\cite{MR4332672} have extended it so as to produce
$p$-adic lifts of mod-$p$ representations into quite general reductive groups
(provided, roughly, that $p$ is sufficiently large and that the ``numerical
coincidence'' mentioned in the previous paragraph holds).
  Their results in particular apply to essentially self-dual $n$-dimensional representations (for~$p$ sufficiently large, depending on the totally real field~$F$).
\end{remark}

\begin{remark}
  One expects that the lifts constructed by Theorem~\ref{thm:main-intro} are pure, but since we do not know that they are potentially modular, we do not know how to show this.
  Consequently, we do not know that the corresponding local lifts lie on a single component of the deformation ring, although of course we expect this to be the case.
\end{remark}
\subsection{Acknowledgements}
We would like to thank Florian Herzig, Vytas Paškūnas and Jack Thorne for their helpful comments on an earlier version of this paper.
In particular, we would like to thank Paškūnas for pointing out the reference~\cite{MR3968877}, and for suggesting that we include the case $p=2$.
We used AI models to proofread the paper.

\section{Deformation rings}\label{sec:culjwvzdmn}
\subsection{Definitions}
\label{sec:definitions}
Let~$F$ be a totally real field.
Let~$\cO$ be the ring of integers in a finite extension~$E/\Qp$, with uniformizer~$\varpi$ and residue field~$\cO/\varpi=\F$.
We assume without further comment that~$E$ is chosen to be sufficiently large, so that for example the image of each embedding $F\into\Qpbar$ is contained in~$E$, and the irreducible components of the deformation rings considered below are geometrically irreducible.

Let $\rhobar:\Gal_F\to\GL_2 (\F)$ be absolutely irreducible and totally odd, in the sense that $\det\rhobar(c)=-1$ for each complex conjugation $c\in\Gal_{F}$.
Let $\psi: \Gal_F\to \cO^\times$ denote a character whose reduction modulo $\varpi$
--- which we denote by $\psibar$ --- is equal to~$\det \rhobar$,
and which is de Rham at all places lying over~$p$, and is totally odd, in that $\psi(c)=-1$ for each complex conjugation $c\in\Gal_{F}$.
For each finite place~$v$ of~$F$, let~$R_v$ denote the universal lifting $\cO$-algebra for liftings of~$\rhobar|_{\Gal_{F_v}}$,
and let~$R_v^{\psi}$ denote the quotient of~$R_{v}$ corresponding to liftings of determinant~$\psi|_{\Gal_{F_{v}}}$.

\begin{defn}\label{defn:Hodge-type}
  If~$v|p$ is a place of~$F$, then a \emph{regular Hodge type} at~$v$
is a tuple $\lambda_v=\{\lambda_{\sigma,1},\lambda_{\sigma,2}\}_{\sigma:F_v\into\Qpbar}$, where $\lambda_{\sigma,1}>\lambda_{\sigma,2}$ are distinct integers. We say that~$\lambda_v$ has \emph{weight 2} if $\{\lambda_{\sigma,1},\lambda_{\sigma,2}\}=\{1,0\}$ for all~$\sigma$, in which case we write $[0,1]_{v}$ for~$\lambda_v$.
\end{defn}

Given a finite set~$\Sigma$ of finite places of~$F$, containing all places at which either
$\rhobar$ or~$\psi$ is ramified and all places lying above~$p$,
we let~$R_{\Sigma}^{\univ}$ denote the universal deformation $\cO$-algebra for deformations of~$\rhobar$ which are unramified at all finite places of~$F$ not contained in~$\Sigma$ (recall that if~$R$ is a complete local $\cO$-algebra with residue field~$\F$, a \emph{deformation} of~$\rhobar$ to~$R$ is an equivalence class of liftings under conjugation by $\ker\bigl(\GL_2 (R)\to\GL_2 (\F)\bigr)$).
We let~$R_{\Sigma}^{\univ,\psi}$ denote the quotient of~$R_{\Sigma}^{\univ}$ corresponding to deformations with determinant~$\psi$.

\begin{defn}\label{defn:global-deformation-problem}
  A \emph{global deformation problem} $\cS$ is a tuple \[\cS=\bigl(\psi,\Sigma,\{\Rbar_v^{\psi}\}_{v\in \Sigma}\bigr)\] consisting of the following data:
  \begin{itemize}[wide]
  \item As above, $\psi:\Gal_F\to\cO^\times$ is a de Rham character such that $\psibar = \det \rhobar$.
  \item $\Sigma$, as above, is a finite set of places of~$F$, containing all the places above~$p$, and all the places at which~$\rhobar$ or~$\psi$ is ramified.
  \item
For each~$v\nmid p$ with $v\in \Sigma$, we let $\Rbar_v^{\psi}$ be a quotient of~$R_v^{\psi}$ corresponding to a (non-empty) union of irreducible components of~$R_v^{\psi}[1/p]$.
  \item For each~$v|p$, we let~$\lambda_v$ be a regular Hodge--Tate type, we let~$E_v/F_v$ be a finite extension, and we write~$R^{\lambda_v,E_v,\psi}_{v}$ for the $\cO$-flat quotient of~$R_v^{\psi}$ corresponding to potentially semistable lifts of~$\rhobar|_{\Gal_{F_{v}}}$ which become semistable of Hodge--Tate type~$\lambda_v$ over~$E_{v}$, as constructed by Kisin~\cite{MR2373358} (see also \cite[Thm.\
A]{zbMATH07042062}). 
We assume that~$R^{\lambda_v,E_v,\psi}_{v}$ is non-zero, and we let $\Rbar_v^{\psi}$ be a quotient of~$R^{\lambda_v,E_v,\psi}_{v}$ corresponding to a non-empty union of irreducible components
    of~$R^{\lambda_v,E_v,\psi}_{v}[1/p]$.
  \end{itemize}
\end{defn}
We let $R_{\cS}$ be the quotient of~$R_{\Sigma}^{\univ,\psi}$ given by the universal deformation $\cO$-algebra for deformations of~$\rhobar$ of determinant~$\psi$, which are unramified at all finite places of~$F$ not contained in~$\Sigma$, and whose restrictions to each~$\Gal_{F_v}$, $v\in \Sigma$ correspond to points of $\Rbar_v^{\psi}$.

Set $R^{\loc}\coloneq\widehat{\otimes}_{v\in \Sigma}R_v$, and $R^{\loc}_{\cS}\coloneq\widehat{\otimes}_{v\in \Sigma}\Rbar_v^{\psi}$, a quotient of~$R^{\loc}$.
If we choose a representative for the universal deformation of~$\rhobar$ to~$\GL_{2}(R_{\Sigma}^{\univ})$, we may regard~$R_{\Sigma}^{\univ}$ as an $R^{\loc}$-algebra, and $R_{\cS}$ as an $R^{\loc}_{\cS}$-algebra.
From the definitions we see that
\begin{equation}
  \label{eq:global-over-local}
  R_{\cS}=R_{\Sigma}^{\univ,\psi}\otimes_{R^{\loc}} R_{\cS}^{\loc}.
\end{equation}

The following dimension bound follows from a standard tangent-obstruction calculation using Galois cohomology, which goes back to Mazur~\cite{MR1012172}, and in this form is essentially due to Kisin~\cite{MR2459302}.

\begin{prop}\label{prop:dim-lower-bound}
Suppose that~$\rhobar|_{\Gal_{F(\zeta_p)}}$ is absolutely irreducible, and if~$p=2$, assume that $\rhobar$ has non-solvable image. Then $\dim R_{\cS}~\ge~1$.
  \end{prop}
  \begin{proof}
If~$p>2$ and~$\lambda_{v}$ has weight~$2$, this is~\cite[Prop.\ 6.2.1]{0905.4266}, and the general case follows from an identical argument, using~\cite[Thm.\ 3.3.4]{MR2373358} for the dimensions of the rings $\Rbar_{v}^{\psi}[1/p]$ for~$v|p$.
Note that when~$p=2$, the assumption that $\rhobar$ has non-solvable image guarantees the vanishing of the group $H^{0}(\operatorname{Gal} _{F,S},(\operatorname{ad}^0 V)^*(1))$ considered in~\cite[Prop.\ 4.1.5]{MR2459302}; see \cite[Lem.\ 4.3 2(ii)]{KW2}.
  \end{proof}
\begin{defn}\label{defn:lift-of-type-S}
  We say that a lift~$\rho:\Gal_{F,\Sigma}\to\GL_{2}(\Zpbar)$ of~$\rhobar\otimes_{\F}\Fpbar$ is of type~$\cS$ if it has determinant~$\psi$, and if for each place~$v\in \Sigma$, the corresponding homomorphism $R_v^{\psi}\to\Zpbar$ factors through $\Rbar_v^{\psi}$.
  By definition, the isomorphism classes of lifts of type~$\cS$ are in bijection with the set of $\cO$-algebra homomorphisms $R_{\cS}\to\Qpbar$.
\end{defn}

\section{Comparing the special fibres of potentially semistable deformation rings}
\label{sec:comp-spec-fibr}

\subsection{Moduli stacks of \'etale \texorpdfstring{$(\varphi,\Gamma)$}{(phi,Gamma)}-modules}

We now make use of the moduli stacks of \'etale
\texorpdfstring{$(\varphi,\Gamma)$}{(phi,Gamma)}-modules constructed in~\cite{emertongeepicture}, or more precisely of the fixed-determinant variant of~\cite[App.\ C]{dotto2026categoricalpadiclanglandscorrespondence}.
Fix a place $v|p$ of~$F$, and let~$\psi:\Gal_{F_v}\to\cO^{\times}$ be a de Rham character.
We write~$\cX^{\psi}$ for the moduli stack of \'etale $(\varphi,\Gamma)$-modules defined in~\cite[Defn.\ C.1.1]{dotto2026categoricalpadiclanglandscorrespondence}, taking the field~$K$ to be~$F_{v}$ and the positive integer~$d$ to be~$2$; this is a formal algebraic stack over~$\Spf\cO$, with the property that the groupoid $\cX^{\psi}(\Spf\cO)$ is naturally equivalent to the groupoid of continuous representations $\Gal_{F_v}\to\GL_2 (\cO)$ equipped with an identification of their determinant with~$\psi$.
The special fibre~$\cX^{\psi}/\varpi$ only depends on~$\psibar$, and if~$\F'/\F$ is a finite extension, $\cX^{\psi}(\Spec \F')$ corresponds to continuous representations $\Gal_{F_v}\to\GL_2 (\F')$ with determinant~$\psibar$.

For each regular Hodge type~$\lambda_v$ and finite extension~$E_v/F_v$, we write~$\cX^{\lambda_v,E_v,\psi}$ for the closed substack of~$\cX^{\psi}$ corresponding to potentially semistable representations which become semistable   of Hodge--Tate type~$\lambda_v$ over~$E_{v}$.
By~\cite[Prop.\ C.1.5, Thm.\ C.1.9]{dotto2026categoricalpadiclanglandscorrespondence} (and the obvious variants for potentially semistable stacks, whose proofs are identical), the underlying reduced substack $\cX^{\psi}_{\red}$ of~$\cX^{\psi}$ is a finite type algebraic stack over~$\F$, and is equidimensional of dimension~$[F_v:\Qp]$.
Furthermore, the special fibre~$\cX^{\lambda_v,E_v,\psi}/\varpi$ of each~$\cX^{\lambda_v,E_v,\psi}$ is also a finite type algebraic stack over~$\F$, and is again equidimensional of dimension~$[F_v:\Qp]$.
In particular, the underlying reduced substack of~$\cX^{\lambda_v,E_v,\psi}$ is a union of irreducible components of~$\cX^{\psi}_{\red}$.

\begin{defn}\label{defn:pot-weight-2-notation}
  We let~$\varepsilon$ denote the $p$-adic cyclotomic character; our convention is that it has Hodge--Tate weight~$-1$.
Let~$\psi_1\coloneq \varepsilon^{-1}\widetilde{(\varepsilonbar\det\rhobar)}$, where the tilde denotes the Teichm\"uller lift; in local contexts 
we also write~$\psi_1$ for its restriction to~$\Gal_{F_v}$. Write~$K_v\coloneq F_v(\varpi_v^{1/(q_v-1)})$, where~$\varpi_v$ is a uniformizer of~$F_v$ and~$q_v$ denotes the cardinality of the residue field of~$\cO_{F_v}$.
\end{defn}
By definition,~$\cX^{[0,1]_v,K_v,\psi_1}$ is the closed substack of~$\cX^{\psi_1}$ corresponding to representations which become semistable over~$K_v$ and are of weight~$2$.
The reason for our choice of~$K_v$ is the following proposition.

\begin{prop}
  \label{prop:underlying-reduced-tame-weight-0}
  The special fibre $\cX^{[0,1]_v,K_v,\psi_1}/\varpi$ contains~$\cX^{\psi_1}_{\red}$.
\end{prop}
\begin{proof}
Since $\cX^{[0,1]_v,K_v,\psi_1}/\varpi$ and~$\cX^{\psi_1}_{\red}$ are of finite type over~$\F$, it suffices to show that for each finite extension~$\F'/\F$, each $\rbar:\Gal_{F_v}\to\GL_2 (\F')$ of determinant~$\psibar_1$ admits a lift to a weight~$2$ representation which becomes semistable over~$K_v$.
  At least when $p>2$, this is an immediate consequence of the Breuil--M\'ezard conjecture; see e.g.\ \cite[\S 8.6]{emertongeepicture}.  We can also prove it (for all~$p$) more directly as follows.
  Since both substacks are closed in~$\cX^{\psi_1}/\varpi$, it suffices to show that $\cX^{[0,1]_v,K_v,\psi_1}/\varpi$ contains a dense set of closed points of~$\cX^{\psi_1}_{\red}$.  Any sufficiently generic closed point of~$\cX^{\psi_1}_{\red}$ corresponds to a reducible~$\rbar$.  It therefore suffices to show that if $\rbar\cong
  \begin{pmatrix}
    \chibar_1 & *\\ 0&\chibar_2\varepsilonbar^{-1} 
  \end{pmatrix}
  $ then~$\rbar$ admits a lift of the form $\begin{pmatrix}
    \widetilde{\chibar_1} & *\\ 0&\widetilde{\chibar_2}\varepsilon^{-1} 
  \end{pmatrix}$ where~$\widetilde{\chibar_i}$  is the Teichm\"uller lift of~$\chibar_i$ (such a lift is automatically semistable over~$K_{v}$).
  The obstruction to the existence of such a lift lies in $H^2(\Gal_{F_v},\cO(\widetilde{\chibar_1\chibar_2^{-1}}\varepsilon))[\varpi]=0$.
\end{proof}
\begin{rem}\label{rem:local-n-dimensional}
  In fact, by a straightforward induction on~$n$ one can prove an analogue of Proposition~\ref{prop:underlying-reduced-tame-weight-0} for~$\GL_n$, replacing~$[0,1]_{v}$ with any regular Hodge type for~$\GL_n$. 
\end{rem}

\subsection{Versal rings}\label{subsec:versal-rings-to-stack}
Proposition~\ref{prop:underlying-reduced-tame-weight-0} has the following corollary related to potentially crystalline/semistable
lifting rings.
To make sense of its statement, note that if~$\psibar=\psibar_1=\det\rhobar$,
we may canonically identify~$R_v^{\psi}/\varpi$ with $R_v^{\psi_1}/\varpi$.

\begin{cor}
  \label{cor:nilpotent-thickening-comparison-local-deformation-rings}
For each~$\lambda_{v},E_v,\psi$ with~$\psibar=\det(\rhobar|_{\Gal_{F_v}})$, there is a nilpotent ideal~$I$ in $R_v^{\lambda_v,E_v,\psi}/\varpi$ with the property that the natural surjection $R_v^{\psi}\onto R_v^{\lambda_v,E_v,\psi}/(\varpi,I)$ factors
through~$R_v^{[0,1]_v,K_v,\psi_1}/\varpi$.
\end{cor}
\begin{proof}
The restriction~$\rhobar|_{\Gal_{F_v}}$ corresponds to a point $x \in \cX^{\psi}(\F)$,
and there is an induced morphism $\Spf R_{v}^{\psi} \to \cX^{\psi}$ which is versal to~$\cX^{\psi}$ at~$x$
(in a precise sense which we do not need to recall here; see~\cite[Prop.\ 3.6.3]{emertongeepicture}).
By~\cite[Prop.~C.1.5~(2)]{dotto2026categoricalpadiclanglandscorrespondence}, this morphism fits into a
fibre product diagram 
$$\xymatrix{
\Spf R_v^{\lambda_v,E_v,\psi}
\ar[r]\ar[d] & \Spf R_v^{\psi} \ar[d] \\
\cX^{\lambda_v,E_v,\psi} \ar[r] & \cX^{\psi}}
$$  

Since $\cX^{\lambda_v,E_v,\psi}/\varpi$ is an algebraic stack of finite type over~$\F$,
the coherent ideal sheaf $\cI$ on~$\cX^{\lambda_v,E_v,\psi}/\varpi$ that cuts out
its closed substack $\cX^{\lambda_v,E_v,\psi}_{\red}$ is nilpotent.
Thus, this latter closed substack pulls back 
to a closed formal subscheme of~$\Spf R_v^{\lambda_v,E_v,\psi}/\varpi$
cut out by a nilpotent ideal sheaf.
Expressing this more concretely, it pulls back to the~$\Spf$ of some 
quotient~$R_v^{\lambda_v,E_v,\psi}/(\varpi,I)$ of~$R_v^{\lambda_v,E_v,\psi}/\varpi$ by a nilpotent ideal~$I$.
The corollary is now immediate from Proposition~\ref{prop:underlying-reduced-tame-weight-0}
(taking into account the analogue of the above fibre product diagram for the ring~$R_v^{[0,1]_v,K_v,\psi_1}/\varpi$).
  \end{proof}

\section{Finiteness of global deformation rings}
\label{sec:finit-global-deform-rings}

\subsection{Finiteness of weight two deformation rings}
\label{sec:finit-potent-bars}
\begin{defn}\label{defn:exceptional}
  Following~\cite{MR3648503}, we say that $\rhobar$ is \emph{exceptional} if~$p=5$, the
  projective image~$\Proj(\rhobar)$ of~$\rhobar$ is isomorphic to~$\PGL_2(\F_5)$, and
  $\overline{F}^{\ker\Proj(\rhobar)}$ contains~$\zeta_5$.
\end{defn}

\begin{prop}
  \label{prop:pot-BT-deformation-ring-finite}
Suppose that~
$\rhobar$ and~$\psi$ are totally odd, and that $\rhobar|_{\Gal_{F(\zeta_{p})}}$ is absolutely irreducible.
If~$p=2$, assume that~$\rhobar$ has non-solvable image, and if ~$p=5$, assume  that~$\rhobar$ is not exceptional in the sense of Definition~\emph{\ref{defn:exceptional}}.

  Assume furthermore that~$\lambda_v$ has weight~$2$ for each~$v|p$.
    Then $R_{\cS}$ is a finite $\cO$-algebra.
  \end{prop}
\begin{proof} Since each $\Rbar_v^{\psi}[1/p]$ has only finitely many irreducible components, we may assume that each~$\Rbar_v^{\psi}[1/p]$ is irreducible.
  If~$p>2$, then bearing in mind~\cite[Prop.\ 7.3.1, 7.4.1]{0905.4266}, the result is then a special case of 
  \cite[Thm.\ 6.1.1]{0905.4266}. (Note that Snowden's
  condition~(A2) is slightly stronger than our assumption
  that~$\rhobar$ is not exceptional if~$p=5$, but this is due to the
  corresponding condition in the version of~\cite{KisinModularity} that is
  cited in~\cite{0905.4266}; the published version
  of~\cite{KisinModularity} assumes only that~$\rhobar$ is not exceptional.) If~$p=2$ then a very similar argument works, appealing to~\cite{KisinTwoAdic} instead of~\cite{KisinModularity}; the analogue for $p=2$ of \cite[Thm.\ 5.1.1]{0905.4266} can be proved in exactly the same way as \cite[Prop.\ 6.7]{MR3598803} (noting that in the proof there, for each place $v|2$, one can choose whether to make the $N$-HBAV $A/L_v$ have  good ordinary reduction or good non-ordinary reduction).
  \end{proof}

\subsection{The main theorem}
\label{sec:main-theorem}
\begin{thm}
  \label{thm:main-thm-existence-lifts}
Suppose that~$\rhobar$ and~$\psi$ are totally odd and that $\rhobar|_{\Gal_{F(\zeta_{p})}}$ is absolutely irreducible. If~$p=2$, assume that~$\rhobar$ has non-solvable image, and if ~$p=5$, assume  that~$\rhobar$ is not exceptional in the sense of Definition~\emph{\ref{defn:exceptional}}.

Then for any global deformation problem~$\cS$, the universal deformation ring~$R_{\cS}$ is a finite~$\cO$-algebra of positive $\cO$-rank.
  In particular, the set of isomorphism classes of lifts of~$\rhobar$ of type~$\cS$ is finite and non-empty.
\end{thm}
\begin{proof}
  In order to see that~$R_{\cS}$ is a finite~$\cO$-algebra, it suffices to show that $R_{\cS}/\varpi$ is a finite $\F$-algebra.
    We now consider the deformation problem \[\cS'=\bigl(\psi_1,\Sigma,\{\Rbar_v^{\psi_1}\}_{v\in \Sigma}\bigr),\] where
  \begin{itemize}[wide]
  \item $\psi_1$ is as in Definition~\ref{defn:pot-weight-2-notation}.
  \item $\Rbar_{v}^{\psi_1}=R_{v}^{\psi_1}$ for $v\nmid p$.
  \item $\Rbar_v^{\psi_1}=R_v^{[0,1]_v,K_v,\psi_1}$ if~$v|p$, where~$K_v$ is as in Definition~\ref{defn:pot-weight-2-notation}.
  \end{itemize}
  For~$v\nmid p$, the canonical identification $R_v^{\psi_1}/\varpi=R_v^{\psi}/\varpi$ and the quotient map $R_v^{\psi}/\varpi\onto\Rbar_v^{\psi}/\varpi$, together with Corollary~\ref{cor:nilpotent-thickening-comparison-local-deformation-rings} for~$v|p$, give a nilpotent ideal $I^{\loc}$ of $R_{\cS}^{\loc}/\varpi$ and a surjection of $R^{\loc}$-algebras $R^{\loc}_{\cS'}/\varpi\onto R_{\cS}^{\loc}/(\varpi,I^{\loc})$.
  Let~$I$ be the ideal of~$R_{\cS}/\varpi$ generated by the image of~$I^{\loc}$ under the map $R_{\cS}^{\loc}/\varpi\to R_{\cS}/\varpi$.
  By~\eqref{eq:global-over-local}, this induces a surjection of $\F$-algebras $R_{\cS'}/\varpi\onto R_{\cS}/(\varpi,I)$.
    By Proposition~\ref{prop:pot-BT-deformation-ring-finite}, $R_{\cS'}/\varpi$ is a finite $\F$-algebra, so the same is true of $R_{\cS}/(\varpi,I)$.
  Since $R_{\cS}$ is Noetherian and~$I$ is nilpotent, we deduce that $R_{\cS}/\varpi$ is a finite $\F$-algebra, as claimed.

  By Proposition~\ref{prop:dim-lower-bound}, $R_{\cS}$ has dimension at least~$1$, so it has dimension precisely~$1$, and has positive $\cO$-rank.
  Thus $R_{\cS}[1/p]\otimes_E\Qpbar$ is a non-zero finite $\Qpbar$-algebra, so the set of $\cO$-homomorphisms $R_{\cS}\to\Qpbar$ is finite and non-empty, and we are done (see Definition~\ref{defn:lift-of-type-S}).
  \end{proof}

\printbibliography

\end{document}